\documentclass[12pt,a4paper]{amsart}

\usepackage{fullpage} 

\usepackage{amssymb}
\usepackage{amsrefs}
\usepackage[all]{xy}

\let\Im=\undefined
\let\mod=\undefined
\DeclareMathOperator{\GL}{GL}
\DeclareMathOperator{\Im}{Im}
\DeclareMathOperator{\rk}{rk}
\DeclareMathOperator{\Ext}{Ext}
\DeclareMathOperator{\Hom}{Hom}
\DeclareMathOperator{\Ker}{Ker}
\DeclareMathOperator{\mod}{mod}
\DeclareMathOperator{\rep}{rep}
\DeclareMathOperator{\supp}{supp}
\DeclareMathOperator{\Coker}{Coker}

\newcommand{\bd}{\mathbf{d}}
\newcommand{\be}{\mathbf{e}}
\newcommand{\bone}{\mathbf{1}}
\newcommand{\bzero}{\mathbf{0}}

\newcommand{\bbA}{\mathbb{A}}
\newcommand{\bbM}{\mathbb{M}}
\newcommand{\bbN}{\mathbb{N}}
\newcommand{\bbZ}{\mathbb{Z}}

\newcommand{\calE}{\mathcal{E}}
\newcommand{\calH}{\mathcal{H}}
\newcommand{\calM}{\mathcal{M}}
\newcommand{\calU}{\mathcal{U}}
\newcommand{\calX}{\mathcal{X}}
\newcommand{\calY}{\mathcal{Y}}

\newtheorem{theo}{Theorem}[section]
\newtheorem{coro}[theo]{Corolary}
\newtheorem{lemm}[theo]{Lemma}
\newtheorem{prop}[theo]{Proposition}

\title[Two vertex geometrically irreducible algebras]{Two vertex geometrically irreducible algebras}

\author{Grzegorz Bobi\'nski}

\address{Grzegorz Bobi\'nski\newline
Faculty of Mathematics and Computer Science,\newline
Nicolaus Copernicus University\newline
ul. Chopina 12/18\newline
87-100 Toru\'n\newline
Poland}
\email{gregbob@mat.umk.pl}

\author{Grzegorz Zwara}

\address{Grzegorz Zwara\newline
Faculty of Mathematics and Computer Science\newline
Nicolaus Copernicus University\newline
ul. Chopina 12/18\newline
    87-100 Toru\'n\newline
Poland}
\email{gzwara@mat.umk.pl}


\begin{document}

\maketitle

\begin{abstract}
We complete a classification of the two-vertex geometrically irreducible algebras. We also classify the algebras in new classes of hom- and ext-irreducible algebras.
\end{abstract}

\section{Introduction and main result} \label{sect:intro}

Throughout the paper $\Bbbk$ denotes an algebraically closed field of arbitrary characteristic.

Given a finite dimension $\Bbbk$-algebra $A$ and a nonnegative integer $d$ we denote by $\mod_A (d)$ the variety of $d$-dimensional $A$-modules, which consists of the $\Bbbk$-algebra homomorphisms $A \to \bbM_{d \times d} (\Bbbk)$ and parameterizes $A$-modules of dimension $d$. The algebra $A$ is called geometrically irreducible if the connected components of $\mod_A (d)$ are irreducible, for all $d$. Geometrically irreducible algebras have been studied in~\cites{BobSch2019a, BobSch2019b, Bob2021a}. In particular, it is conjectured in~\cite{BobSch2019b} that every geometrically irreducible algebra is a gluing (in a sense explained in~\cite{BobSch2019b}) of algebras with at most two simples. This conjecture has been confirmed for the algebras without shortcuts (\cite{BobSch2019b}*{Theorem~1.2}). Consequently, the geometrically irreducible algebras with at most two simples seem to be of particular interest. The aim of this paper is to complete the classification of this class of algebras initiated in~\cites{BobSch2019b, Bob2021a}.

We introduce some families of algebras. For $n \geq 0$, let $Q (n)$ be the quiver
\[
\xymatrix{0 \ar@(ul,dl)_{\varepsilon_0} & 1 \ar@(ur,dr)^{\varepsilon_1} \ar@<-1.5ex>_{\alpha_1}[l] \ar@<1.5ex>^{\alpha_n}_{\vdots}[l]}.
\]
For $l \geq 1$ (and $n \geq 1$), let
\[
\rho^{(l)} := \sum_{i = 0}^l \varepsilon_0^{l - i} \alpha_1 \varepsilon_1^i.
\]
For $m \geq 2$ and $n, l \geq 1$, let $A (n, m, l)$ be the path algebra of the quiver $Q (n)$ bounded by $\varepsilon_0^m$, $\varepsilon_1^m$ and $\rho^{(l)}$. Similarly, for $m_0, m_1 \geq 1$ and $n \geq 0$, let $A' (n, m_0, m_1)$ be the path algebra of the quiver $Q (n)$ bounded by $\varepsilon_0^{m_0}$ and $\varepsilon_1^{m_1}$.

The following theorem is the first main result of the paper.

\begin{theo} \label{theo:main}
Let $A$ be an algebra which has exactly two simples. Then $A$ is geometrically irreducible if and only if $A$ is Morita equivalent to one of the following algebras:
\begin{enumerate}

\item
$A (n, m, 1)$, for some $m \geq 2$ and $n \geq 1$, or

\item
$A (n, m, m - 1)$, for some $m \geq 2$ and $n \geq 1$, or

\item
$A' (n, m_0, m_1)$, for some $m_0, m_1 \geq 1$ and $n \geq 0$.

\end{enumerate}
\end{theo}

The main result of~\cite{BobSch2019b} (Theorem~1.1) states that if $A$ is a geometrically irreducible algebra with exactly two simples, then (up to Morita equivalence) $A$ is one of the algebras listed in Theorem~\ref{theo:main}. On the other hand, it is relatively easy to see that the algebras $A' (n, m_0, m_1)$, $m_0, m_1 \geq 1$ and $n \geq 0$, are geometrically irreducible (see for example~\cite{BobSch2019b}*{Proposition~3.3}). Moreover, \cite{Bob2021a}*{Corollary~2.1} implies that the algebras $A (n, m, 1)$, $m \geq 2$ and $n \geq 1$, are also geometrically irreducible. Thus in order to prove Theorem~\ref{theo:main} it suffices to prove the following.

\begin{theo} \label{theo:mainbis}
If $m \geq 2$ and $n \geq 1$, then the algebra $A (n, m, m - 1)$ is geometrically irreducible.
\end{theo}

The paper is organized as follows. In Section~\ref{sect:pre} we recall basic definitions, while in Section~\ref{sect:heirr} we introduce classes of hom- and ext-irreducible algebras, which play crucial role in our proof of the main result presented in Section~\ref{sect:proof}. The final Section~\ref{sect:class} contains a classification of the hom- and ext-irreducible algebras.

The both authors gratefully acknowledge the support of the National Science Centre grant no.~2020/37/B/ST1/00127.

\section{Preliminaries} \label{sect:pre}

Throughout the paper, by $\bbZ$, $\bbN$, and $\bbN_+$ we denote the sets of integers, nonnegative integers, and positive integers, respectively. If $i, j \in \bbZ$, then $[i, j]$ denotes the set of all $l \in \bbZ$ such that $i \leq l \leq j$.

A quiver $Q$ consists of a finite set $Q_0$ of vertices, a finite set $Q_1$ of arrows, and two maps $s, t \colon Q_1 \to Q_0$, which assign to an arrow
$\alpha \in Q_1$ its starting and terminating vertex, respectively. An arrow $\alpha \in Q_1$ is called a loop if $s \alpha = t \alpha$. If $\alpha$ is an arrow in a quiver $Q$, then we define the degree $\deg \alpha$ of $\alpha$ by $\deg \alpha := 0$, if $\alpha$ is a loop, and $\deg \alpha := 1$, otherwise. We denote by $Q_1^{(i)}$ the set of arrows of degree $i$, $i = 0, 1$ (in particular, $Q_1^{(0)}$ is the set of loops in $Q$). By a path of length $l \in \bbN_+$ in $Q$ we mean a sequence $\sigma = \alpha_1 \cdots \alpha_l$ of arrows such that $s \alpha_i = t \alpha_{i + 1}$, for all $i \in [1, l - 1]$. We put $s \sigma := s \alpha_l$ and $t \sigma := t \alpha_1$, and call them the starting and the terminating vertex of $\sigma$, respectively. By the degree $\deg \sigma$ of $\sigma$ we mean $\sum_{i = 1}^l \deg \alpha_i$. By an oriented cycle we mean a path $\sigma$ of positive length such that $s \sigma = t \sigma$. For each vertex $x$ of $Q$ we also consider the trivial path $\bone_x$ at vertex $x$ (i.e., $s \bone_x := x =: t \bone_x)$ of length (and degree) $0$. If $X \subseteq Q_0$, then we put $\bone_X := \sum_{x \in X} \bone_x$.

With a quiver $Q$ we associate its path algebra $\Bbbk Q$, which as a $\Bbbk$-vector space has a basis formed by all paths in $Q$ and whose multiplication is induced by the composition of paths. By a relation $\rho$ in $Q$ we mean a linear combination of paths of length at least $2$ with the same starting and terminating vertex (denoted by $s \rho$ and $t \rho$, respectively). If $\rho = \sum_{i = 1}^n \lambda_i \sigma_i$ is a relation, where $\lambda_1$, \ldots, $\lambda_n$ are nonzero scalars and $\sigma_1$, \ldots, $\sigma_n$ pairwise different paths, the by the degree $\deg \rho$ of $\rho$ we mean $\min \{ \deg \sigma_i : \text{$i \in [1, n]$} \}$. By a bound quiver we mean a pair $(Q, I)$ consisting of a quiver $Q$ and an ideal $I$ of $\Bbbk Q$ generated by relations, such that there exists $n \in \bbN_+$ with $\sigma \in I$, for each path $\sigma$ of length (at least) $n$. If $(Q, I)$ is a bound quiver, then we call $\Bbbk Q / I$ the path algebra of $(Q, I)$. If $A$ is the path algebra of a bound quiver $(Q, I)$ and the ideal $I$ is generated by relations $\rho_1$, \ldots, $\rho_n$, then we also say that $A$ is the path algebra of $Q$ bounded by $\rho_1$, \ldots, $\rho_n$.

If $A$ is an algebra, then there exists a bound quiver $(Q, I)$ such that $A$ and $\Bbbk Q / I$ are Morita equivalent. Moreover, it follows from~\cite{Bon1991} that $A$ is geometrically irreducible if and only if $\Bbbk Q / I$ is geometrically irreducible.

Thus from now on we concentrate on path algebras of bound quivers. The quiver $Q$ is uniquely determined by $A$ and is called the Gabriel quiver of $A$. The algebra $A$ is called weakly triangular, if there are no oriented cycles of positive degree in $Q$.

Let $(Q, I)$ be a bound quiver. By a representation of $(Q, I)$ we mean a collection $V$ of finite dimensional $\Bbbk$-vector spaces $V_x$, $x \in Q_0$, together with $\Bbbk$-linear maps $V_\alpha \colon V_{s \alpha} \to V_{t \alpha}$, $\alpha \in Q_1$, such that $V_\rho$ (defined in a natural way) vanishes for all relations $\rho \in I$. The collection $(\dim_k V_x) \in \bbN^{Q_0}$ is called the dimension vector of $V$ (and the elements of $\bbN^{Q_0}$ are called dimension vectors). If $U$ and $V$ are representations of $(Q, I)$, then a homomorphism $f \colon V \to U$ is a collection of $\Bbbk$-linear maps $f_x \colon V_x \to U_x$, $x \in Q_0$, such that $U_\alpha \circ f_{s \alpha} = f_{t \alpha} \circ V_\alpha$, for each $\alpha \in Q_1$. If $A$ is the path algebra of $(Q, I)$ then the category $\mod A$ of (finite dimensional left) $A$-modules is equivalent to the category $\rep (Q, I)$ of representations of $(Q, I)$. By abuse of language we call objects of $\rep (Q, I)$ representations of $A$. If $x \in Q_0$ then we denote by $S_A (x)$ the simple representation at vertex $x$, i.e., $S_A (x)_x := \Bbbk$, $S_A (x)_y := 0$, for $y \neq x$, and $S_A (x)_\alpha := 0$, for $\alpha \in Q_1$.

According to~\cite{Bon1991} the above equivalence has its geometric counterpart. Namely, let $A$ be the path algebra of a bound quiver $(Q, I)$. For a dimension vector $\bd$ we denote by $\rep_A (\bd)$ the set of representations $V$ of $A$ such that $V_x = \Bbbk^{d_x}$, for each $x \in Q_0$. One easily observes that $\rep_A (\bd)$ can be identified with a closed subset of $\prod_{\alpha \in Q_1} \bbM_{d_{t \alpha} \times d_{s \alpha}} (\Bbbk)$, thus carries a structure of an affine variety. Then results of~\cite{Bon1991} imply that $A$ is geometrically irreducible if and only if $\rep_A (\bd)$ is irreducible, for each $\bd \in \bbN^{Q_0}$. More precisely, if $d$ is a nonnegative integer then the connected components of $\mod_A (d)$ are (isomorphic to) the associated fiber bundles $\GL (d) \times_{\GL (\bd)} \rep_A (\bd)$, for the dimension vectors $\bd$ such that $\sum_{x \in Q_0} d_x = d$, where $\GL (\bd) := \prod_{x \in Q_0} \GL (d_x)$.

Let $A$ be the path algebra of a bound quiver $(Q, I)$ and $\bd$ a dimension vector. The group $\GL (\bd)$ acts on $\rep_A (\bd)$ by conjugation, i.e., if $g \in \GL (\bd)$ and $V \in \rep_A (\bd)$, then $(g \ast V)_\alpha := g_{t \alpha} V_\alpha g_{s \alpha}^{-1}$, for $\alpha \in Q_1$. Note that $g \in \GL (\bd)$ can be viewed as an isomorphism $g \colon V \to g \ast V$.

We finish this subsection by listing basic properties of geometrically irreducible algebras proved in~\cite{BobSch2019a}*{Section~3}.

\begin{lemm} \label{lemm:geoirr}
The following hold for a geometrically irreducible algebra $A$ with the Gabriel quiver $Q$.
\begin{enumerate}

\item \label{point:geoirr1}
$A$ is weakly triangular.

\item \label{point:geoirr2}
For each vertex $x \in Q_0$ there exists at most one loop $\alpha \in Q_1$ with $s \alpha = x$.

\item \label{point:geoirr3}
If $|Q_0| = 1$, then $A \simeq k [x] / (x^m)$. \qed

\end{enumerate}
\end{lemm}

\section{Hom-, mono- and ext-irreducible algebras} \label{sect:heirr}

Let $\Lambda$ be an algebra. For dimension vectors $\bd$ and $\be$ we denote by $\calH_\Lambda (\be, \bd)$ the set of triples $(V, W, f)$, such that $V \in \rep_\Lambda (\be)$, $W \in \rep_\Lambda (\bd)$, and $f \in \Hom_\Lambda (V, W)$. Then $\calH_\Lambda (\be, \bd)$ has a natural structure of an affine variety. We denote by $\calM_\Lambda (\be, \bd)$ the subset of $\calH_\Lambda (\be, \bd)$ consisting of the triples $(V, W, f)$ such that $f$ is a monomorphism. Note that $\calM_\Lambda (\be, \bd)$ is an open subset of $\calH_\Lambda (\be, \bd)$. Moreover, $\calM (\be, \bd)$ is nonempty if and only if $\be \leq \bd$ (i.e., $e_x \leq d_x$, for each $x \in Q_0$). We say that $\Lambda$ is hom-irreducible (mono-irreducible) if $\calH_\Lambda (\be, \bd)$ ($\calM_\Lambda (\be, \bd)$, respectively) is irreducible, for all dimension vectors $\bd$ and $\be$ (such that $\be \leq \bd$, respectively). We have the following obvious observation.

\begin{prop} \label{prop:3.1}
If an algebra $\Lambda$ is hom-irreducible, then $\Lambda$ is mono-irreducible. \qed
\end{prop}

In order to define ext-irreducible algebras we need to recall an interpretation of extensions of $\Lambda$-modules in terms of cocycles. Fix the path algebra $\Lambda$ of a bound quiver $(Q, I)$, dimension vectors $\bd$ and $\be$, and $V \in \rep_\Lambda (\bd)$ and $U \in \rep_\Lambda (\be)$. Let $\bbA_Q^{\be, \bd} := \prod_{\alpha \in Q_1} \bbM_{d_{t \alpha} \times e_{s \alpha}} (\Bbbk)$. For each $Z \in \bbA_Q^{\be, \bd}$ we have a representation $W^{V, Z, U}$ of $Q$ defined by: $W^{V, Z, U}_x := \Bbbk^{d_x + e_x} (= \Bbbk^{d_x} \oplus \Bbbk^{e_x})$, for $x \in Q$, and $W^{V, Z, U}_\alpha := \left[
\begin{smallmatrix}
V_\alpha & Z_\alpha \\ 0 & U_\alpha
\end{smallmatrix}
\right]$, for $\alpha \in Q_1$. Given an exact sequence
\[
0 \to V \xrightarrow{f} W \xrightarrow{h} U \to 0
\]
of $\Lambda$-modules, there exists $Z \in \bbA_Q^{\be, \bd}$ and an isomorphism $g \colon W^{V, Z, U} \to W$, such that the following diagram is commutative
\[
\xymatrix{0 \ar[r] & V \ar[r]^-\mu \ar@{=}[d] & W^{V, Z, U} \ar[r]^-\pi \ar[d]^g & U \ar[r] \ar@{=}[d] & 0 \\ 0 \ar[r] & V \ar[r]^f & W \ar[r]^h & U \ar[r] & 0}
\]
where $\mu$ and $\pi$ are the natural maps. Commutativity of the above diagram means in particular that $f = g \circ \mu$. Moreover, if $W \in \rep_\Lambda (\bd + \be)$, then $g \in \GL (\bd + \be)$ and $W = g \ast W^{V, Z, U}$. Let $\bbZ^{U, V}$ be the set of $Z \in \bbA_Q^{\be, \bd}$ such that $W^{V, Z, U} \in \rep_\Lambda (\bd + \be)$. Then $Z \in \bbZ^{U, V}$ if and only if $Z_\rho^{U, V} = 0$, for each relation $\rho \in I$, where
\[
Z_\rho^{U, V} = \sum_{i = 1}^n \sum_{j = 1}^{l_i} \lambda_i V_{\alpha_{i, 1}} \cdots V_{\alpha_{i, j - 1}} Z_{\alpha_{i, j}} U_{\alpha_{i, j + 1}} \cdots U_{\alpha_{i, l_i}},
\]
provided $\rho = \sum_{i = 1}^n \lambda_i \alpha_{i, 1} \cdots \alpha_{i, l_i}$.

Given an algebra $\Lambda$ and dimension vectors $\bd$ and $\be$ we denote by $\calE_\Lambda (\be, \bd)$ the set of triples $(U, V, Z)$ such that $U \in \rep_\Lambda (\be)$, $V \in \rep_\Lambda (\bd)$, and $Z \in \bbZ^{U, V}$. Obviously $\calE_\Lambda (\be, \bd)$ has a structure of an affine variety and we call $\Lambda$ ext-irreducible if the varieties $\calE_\Lambda (\be, \bd)$ are irreducible, for all dimension vectors $\bd$ and $\be$.

The following proposition, whose proof uses a construction from~\cite{Bon1994}, is the main result of this section.

\begin{prop} \label{prop:3.2}
An algebra $\Lambda$ is ext-irreducible if and only if $\Lambda$ is mono-irreducible.
\end{prop}

\begin{proof}
Assume first that $\Lambda$ is ext-irreducible and fix dimension vectors $\bd$ and $\be$ such that $\be \leq \bd$. Consider the map $\Phi \colon \GL (\bd) \times \calE_\Lambda (\bd - \be, \be) \to \calM_\Lambda (\be, \bd)$ given by:
\[
\Phi (g, (U, V, Z)) := (V, g \ast W^{V, Z, U}, g \circ \mu),
\]
for $g \in \GL (\bd)$, $V \in \rep_\Lambda (\be)$, $U \in \rep_\Lambda (\bd - \be)$, and $Z \in \bbZ^{U, V}$, where $\mu \colon V \to W^{V, Z, U}$ is the natural inclusion. One easily checks that $\Phi$ is a well-defined morphism of varieties. We show that $\Phi$ is onto and this will imply that $\calM_\Lambda (\be, \bd)$ is irreducible as an image of an irreducible set, since $\GL (\bd)$ is irreducible and by assumption $\calE_\Lambda (\bd - \be, \be)$ is irreducible.

Let $(V, W, f) \in \calM_\Lambda (\be, \bd)$ and $U := \Coker f$. Without loss of generality we may assume $U \in \rep_\Lambda (\bd - \be)$. We have an exact sequence
\[
0 \to V \xrightarrow{f} W \to U \to 0.
\]
By remarks preceding the proposition, there exists $Z \in \bbZ^{U, V}$ and $g \in \GL (\bd)$ such that $W = g \ast W^{V, Z, U}$ and $f = g \circ \mu$. In other words, $(V, W, f) = \Phi (g, (U, V, Z))$.

Now assume that $\Lambda$ is mono-irreducible and fix dimension vectors $\bd$ and $\be$. Let $\calY$ be the set of tuples $(V, W, f, h)$ such that $(V, W, f) \in \calM_\Lambda (\bd, \bd + \be)$ and $h \in \prod_{x \in Q_0} \bbM_{(d_x + e_x) \times e_x} (\Bbbk)$, and $\calX$ be the subset of $\calY$ consisting of the tuples $(V, W, f, h)$ such that  the collection $[f, h] := ([f_x, h_x])_{x \in Q_0} \in \prod_{x \in Q_0} \bbM_{(d_x + e_x) \times (d_x + e_x)} (\Bbbk)$ belongs to $\GL (\be + \bd)$. Then $\calY$ is a vector bundle over $\calM_\Lambda (\be, \be + \bd)$ and $\calX$ is a nonempty open subset of $\calY$, hence irreducible.

We have obvious maps $\pi_{1, 1} \colon \rep_\Lambda (\bd + \be) \to \bbA_Q^{\bd, \bd}$, $\pi_{1, 2} \colon \rep_\Lambda (\bd + \be) \to \bbA_Q^{\be, \bd}$, $\pi_{2, 1} \colon \rep_\Lambda (\bd + \be) \to \bbA_Q^{\bd, \be}$, and $\pi_{2, 2} \colon \rep_\Lambda (\bd + \be) \to \bbA_Q^{\be, \be}$, where $Q$ is the Gabriel quiver of $\Lambda$. If $(V, W, f, h) \in \calX$, $g := [f, h]$, and $W' := g^{-1} \ast W$, then $\pi_{1, 1} (W') = V$, $\pi_{2, 1} (W') = 0$, and $\pi_{1, 2} (W') \in \bbZ^{U, V}$, where $U := \pi_{2, 2} (W)$. Consequently, we may define $\Psi \colon \calX \to \calE_\Lambda (\be, \bd)$ by
\[
\Psi (V, W, f, h) := (\pi_{2, 2} ([f, h]^{-1} \ast W), V, \pi_{1, 2} ([f, h]^{-1} \ast W)),
\]
for $(V, W, f, h) \in \calX$. Obviously $\Psi$ is a regular map, which is easily seen to be onto, hence $\calE_\Lambda (\be, \bd)$ is irreducible.
\end{proof}

We formulate the following obvious consequence of Propositions~\ref{prop:3.1} and~\ref{prop:3.2}.

\begin{coro} \label{coro:3.3}
If an algebra $\Lambda$ is hom-irreducible, then $\Lambda$ is ext-irreducible. \qed
\end{coro}

In Section~\ref{sect:class} we classify the hom- and  ext-irreducible algebras. In particular, we show that the hereditary algebras, which are not semisimple, are ext-irreducible, but not hom-irreducible.

\section{Proof of the main result} \label{sect:proof}

Throughout this section we fix $m \geq 2$. In order to simplify notation we put $A := A (1, m, 1)$ and $B_n := A (n, m, m - 1)$, for $n \in \bbN_+$. Moreover, by $A'$ we denote the path algebra of the quiver
\[
\xymatrix{0 \ar@(ul,dl)_{\varepsilon_0} & 1 \ar@(ur,dr)^{\varepsilon_1} \ar_\alpha[l]}
\]
bounded by $\varepsilon_0^m$, $\varepsilon_1^m$, and $\varepsilon_0 \alpha - \alpha \varepsilon_1$. Finally, we put $\Lambda := k [x] / (x^m)$. Note that the dimension vectors for $\Lambda$ are the elements of $\bbN$, while the dimension vectors for $A$, $A'$ and $B_n$ are the pairs $(d, e)$ with $d, e \in \bbN$. Furthermore, the correspondence given by
\[
\varepsilon_0 \mapsto \varepsilon_0, \alpha \mapsto \alpha_1, \varepsilon_1 \mapsto - \varepsilon_1,
\]
induces an isomorphism $A' \simeq A$. Since the algebra $A$ is geometrically irreducible by \cite{Bob2021a}*{Corollary~1.2}, $A'$ is geometrically irreducible as well.

The crucial point in the proof is the following.

\begin{lemm} \label{lemm:4.1}
Let $d, e \in \bbN$. Then the following hold.
\begin{enumerate}

\item \label{point:hom}
$\rep_{A'} ((d, e)) \cong \calH_\Lambda (e, d)$.

\item \label{point:ext}
$\rep_{B_1} ((d, e)) \cong \calE_\Lambda (e, d)$.

\end{enumerate}
\end{lemm}

\begin{proof}
\eqref{point:hom}~Note that $\Lambda$ is the path algebra of the quiver $Q$ of the form
\[
\xymatrix{0 \ar@(ur,dr)^\varepsilon}
\]
bounded by $\varepsilon^m$. Consequently, $\rep_\Lambda (d)$ consists of the matrices $V \in \bbM_{d \times d} (\Bbbk)$ such that $V^m = 0$. Similarly, $\rep_\Lambda (e)$ consists of the matrices $U \in \bbM_{e \times e} (\Bbbk)$ such that $U^m = 0$. Finally, if $V \in \rep_\Lambda (d)$ and $U \in \rep_\Lambda (e)$, then $\Hom_\Lambda (U, V)$ consists of the matrices $f \in \bbM_{d \times e} (\Bbbk)$ such that $f U = V f$. Consequently, we have an isomorphism $\rep_{A'} ((d, e)) \cong \calH_\Lambda (e, d)$ given by
\[
\rep_{A'} ((d, e)) \ni M \to (M_{\varepsilon_1}, M_{\varepsilon_0}, M_{\alpha_1}) \in \calH_\Lambda (e, d).
\]

\eqref{point:ext}~Note that $\bbA_Q^{e, d} = \bbM_{d \times e} (\Bbbk)$, while $\calE_\Lambda (e, d)$ consists of the triples $(V, U, Z)$ with $V \in \rep_\Lambda (d)$, $U \in \rep_\Lambda (e)$, and $Z \in \bbA_Q^{e, d}$, such that $Z_{\varepsilon^m}^{U, V} = 0$. By easy induction one shows that $Z^{U, V}_{\varepsilon^l} = \sum_{i = 0}^{l - 1} V^{l - 1 - i} Z U^i$, for $l \in \bbN_+$. Consequently the map
\[
\rep_{B_1} ((d, e)) \ni M \to (M_{\varepsilon_1}, M_{\varepsilon_0}, M_{\alpha_1}) \in \calE_\Lambda (e, d)
\]
is easily seen to be an isomorphism.
\end{proof}

We have the following consequences.

\begin{coro} \label{coro:4.2}
The algebra $\Lambda$ is hom- and ext-irreducible.
\end{coro}

\begin{proof}
If $d, e \in \bbN$, then $\calH_\Lambda (e, d) \cong \rep_{A'} ((d, e))$ by Lemma~\ref{lemm:4.1}\eqref{point:hom}. Since $A'$ is geometrically irreducible, $\calH_\Lambda (e, d)$ is irreducible, thus $\Lambda$ is hom-irreducible. Ext-irreducibility follows from Corollary~\ref{coro:3.3}.
\end{proof}

\begin{coro} \label{coro:4.3}
The algebra $B_1$ is geometrically irreducible.
\end{coro}

\begin{proof}
Fix a dimension vector $(d, e)$. Since $\Lambda$ is ext-irreducible by Corollary~\ref{coro:4.2} and $\rep_{B_1} ((d, e)) \cong \calE_\Lambda (e, d)$ by Lemma~\ref{lemm:4.1}\eqref{point:ext}, $\rep_{B_1} ((d, e))$ is irreducible and the claim follows.
\end{proof}

\begin{proof}[Proof of Theorem~\ref{theo:mainbis}]
Fix $n \in \bbN_+$. If $(d, e)$ is a dimension vector, then
\[
\rep_{B_n} ((d, e)) \cong \rep_{B_1} ((d, e)) \times \prod_{i = 2}^n \bbM_{d \times e} (\Bbbk).
\]
Since $B_1$ is geometrically irreducible (Corollary~~ \ref{coro:4.3}), $\rep_{B_n} ((d, e))$ is irreducible, thus $B_n$ is geometrically irreducible.
\end{proof}

\section{Classification results} \label{sect:class}

It is natural to ask about a classification of the hom- and ext-irreducible algebras. We solve this problem in this section.

We start with the following remark. One could define hom- and ext-irreducibility for not necessarily basic algebras, using module varieties instead of representation varieties
-- as usual one requires that the connected components of corresponding varieties are irreducible. An analysis of these constructions, based on results of~\cite{Bon1991}, shows that notions of hom- and ext-irreducibility defined in this way coincides with these defined in Section~\ref{sect:heirr}. More precisely, let $A$ be the path algebra of a bound quiver $(Q, I)$ and $\Lambda$ an algebra Morita equivalent to $A$. Assume also that $S_x$, $x \in Q_0$, form a complete set of pairwise nonisomorphic simple $\Lambda$-modules such that the module $S_x$ corresponds to the vertex $x$, for each $x \in Q_0$. For nonnegative integers $d$ and $e$ we define $\calH_\Lambda' (e, d)$ to be the set of triples $(V, W, f)$, such that $V \in \mod_\Lambda (e)$, $W \in \mod_\Lambda (d)$, and $f \in \Hom_\Lambda (V, W)$. Then one shows (using results of~~\cite{Bon1991}*{Section~2}) that the connected components of $\calH_\Lambda' (e, d)$ are the associated fiber bundles $(\GL (e) \times \GL (d)) \times_{(\GL (\be) \times \GL (\bd))} \calH_A (\be, \bd)$, for the dimension vectors $\bd$ and $\be$ such that $\sum_{x \in Q_0} d_x \cdot \dim_\Bbbk S_x = d$ and $\sum_{x \in Q_0} e_x  \cdot \dim_\Bbbk S_x = e$. A similar statement also holds for ext-varieties, which in the module varieties case are defined in~\cite{Bon1994}*{subsection~2.1}. We leave details to the interested reader.

Consequently, we will work further with the path algebras of bound quivers, but formulate the results for arbitrary algebras.

For $m \in \bbN_+$, let $\Lambda_m := k [x] / (x^m)$. We have the following classification of the hom-irreducible algebras.

\begin{theo}
Up to Morita equivalence, the algebras $\Lambda_m$, $m \in \bbN_+$, are precisely the hom-irreducible connected algebras.
\end{theo}

\begin{proof}
We already know that the algebras $\Lambda_m$, $m \in \bbN_+$, are hom-irreducible (Corollary~\ref{coro:4.2}). It remains to show that they are the only hom-irreducible connected algebras.

Let $(Q, I)$ be a bound quiver such that $A := \Bbbk Q / I$ is a hom-irreducible connected algebra. It is obvious that $A$ is geometrically irreducible ($\rep_A (\bd) = \calH_A (\bzero, \bd)$). Suppose there exists an arrow $\alpha_1 \colon x \to y$ in $Q$ with $x \neq y$. Since $A$ is weakly triangular by Lemma~\ref{lemm:geoirr}\eqref{point:geoirr1}, there are no arrows $y \to x$ in $Q$. Let $\alpha_2, \ldots, \alpha_n \colon x \to y$ be the remaining arrows from $x$ to $y$ in $Q$. Let $\bd$ and $\be$ be the dimension vectors such that
\[
d_z := \delta_{x, z} + \delta_{y, z} \qquad \text{and} \qquad e_z := \delta_{x, z},
\]
for $z \in Q_0$, where $\delta_{u, v}$ is the Kronecker delta.

Fix a loop $\alpha$ in $Q$. Recall that there exists $n \in \bbN_+$ such that $\alpha^n \in I$. Consequently, if $W\in \rep_A (\bd)$, then $W_\alpha^n = 0$. Since $d_z \leq 1$, for each $z \in Q_0$, this implies $W_\alpha = 0$. Similarly, $V_\alpha = 0$, for each $V \in \rep_A (\be)$.

Thus $\calH_A (\be, \bd)$ is isomorphic to the variety
\[
\{ (b, a_1, \ldots, a_n) \in \Bbbk^{n + 1} : \text{$a_1 b = \cdots = a_n b = 0$} \},
\]
which is easily seen to be reducible.

The above implies that $|Q_0| = 1$ (due to connectedness of $A$). Using again geometrical irreducibility of $A$ we get $A \simeq \Lambda_m$, for some $m \in \bbN_+$, by Lemma~\ref{lemm:geoirr}\eqref{point:geoirr3}.
\end{proof}

In order to classify the ext-irreducible algebras, we introduce the corresponding class of algebras. We call an algebra $A$ a loop extension of a hereditary algebra, if $A$ is the path algebra of a quiver $Q$ bound by relations of degree $0$. Note that $A / \langle Q_1^{(0)} \rangle$ is isomorphic to the path algebra of the quiver $(Q_0, Q_1^{(1)})$ (hence hereditary). In other words, $A$ is obtained from a hereditary algebra (i.e., the path algebra of a quiver with no relations) by adding loops and relations involving only these loops. In particular $A$ is weakly triangular. If additionally, for each vertex $x$ of $Q$ there is at most one loop $\alpha$ such that $s \alpha = x$, then we call $A$ a simple loop extension of a hereditary algebra. We show that the ext-irreducible algebras are exactly the simple loop extensions of hereditary algebras.

We start with the following.

\begin{lemm} \label{lemm:ext.irr}
Every hereditary algebra is ext-irreducible.
\end{lemm}

\begin{proof}
Let $A$ be a hereditary algebra. Then $A = \Bbbk Q$, for some quiver $Q$. Consequently, if we fix dimension vectors $\bd$ and $\be$, then $\rep_A (\bd)$ and $\rep_A (\be)$ are affine spaces, and $\bbZ^{U, V} = \bbA_Q^{\be, \bd}$, for all representations $V \in \rep_A (\bd)$ and $U \in \rep_A (\be)$. Consequently, $\calE_A (\be, \bd) = \rep_A (\be) \times \rep_A (\bd) \times \bbA_Q^{\be, \bd}$ is an affine space, thus irreducible.
\end{proof}

As a consequence we obtain the first part of our classification result.

\begin{lemm} \label{prop:ext.irr}
Every simple loop extension of a hereditary algebra is ext-irreducible.
\end{lemm}

\begin{proof}
Let $A = \Bbbk Q / I$ be a simple loop extension of a hereditary algebra. Let $H$ be the path algebra of the quiver $(Q_0, Q_1^{(1)})$. Next, for each loop $\varepsilon$, let $m_\varepsilon$ be the minimal $m$ such that $\varepsilon^m \in I$, and $\Lambda_\varepsilon := k [x] / (x^{m_\varepsilon})$. Note that $I = \langle \varepsilon^{m_\varepsilon} : \text{$\varepsilon \in Q_1^{(0)}$} \rangle$.

Let $\bd$ and $\be$ be dimension vectors. Then
\[
\calE_A (\be, \bd) = \calE_H (\be, \bd) \times \prod_{\varepsilon \in Q_1^{(0)}} \calE_{\Lambda_\varepsilon} (e_{t \varepsilon}, d_{s \varepsilon}).
\]
Since $H$ is ext-irreducible by Lemma~\ref{lemm:ext.irr}, and the algebras $\Lambda_\varepsilon$ are ext-irreducible by Corollary~\ref{coro:4.2}, $\calE_A (\be, \bd)$ is irreducible, and the claim follows.
\end{proof}

In order to prove that the simple loop extensions of hereditary algebras are the only ext-irreducible algebras, we need additional facts on relations in (geometrically irreducible) algebras. Let $Q$ be a quiver. Recall that a set $R$ of relations is called minimal, if $\langle R \rangle \neq \langle R \setminus \{ \rho \} \rangle$, for each $\rho \in R$. If $R$ is a set of relations, then for $x, y \in R$ we denote by $R_{x, y}$ the set of $\rho \in R$ such that $s \rho = x$ and $t \rho = y$. The first fact is the following generalization of~\cite{Bon1983}*{Proposition~1.2}.

\begin{prop} \label{prop:Bongartz}
Let $A = \Bbbk Q / I$ be a weakly triangular algebra. If $R$ is a minimal set of relations generating $I$, then the cardinality $\# R_{x, y}$ of $R_{x, y}$ is independent of $R$, for all $x, y \in Q_0$, $x \neq y$. More precisely,
\[
\# R_{x, y} = \dim_\Bbbk \Ext_A^2 (S_A (x), S_A (y)).
\]
\end{prop}

\begin{proof}
Let $J$ be the ideal in $\Bbbk Q$ generated by the arrows. Then according to~\cite{Bon1983}*{Corollary~1.2}
\[
\dim_\Bbbk \Ext_A^2 (S_A (x), S_A (y)) = \dim_\Bbbk \bone_y (I / (I J + J I)) \bone_x.
\]
Next, let $R^{(0)}$ be the set of relations of degree $0$ in $R$, $A' = \Bbbk Q / \langle R^{(0)} \rangle$, $I' = I / \langle R^{(0)} \rangle$, and $J' = J / \langle R^{(0)} \rangle$. Since $x \neq y$,
\[
\dim_\Bbbk \bone_y (I / (I J + J I)) \bone_x = \dim_\Bbbk \bone_y (I' / (I' J' + J' I')) \bone_x.
\]
Thus in order to finish the proof it suffices to show
\[
\# R_{x, y} = \dim_\Bbbk \bone_y (I' / (I' J' + J' I')) \bone_x.
\]
In other words, it is enough to show that the residue classes of the elements of $R_{x, y}$ form a basis of $\bone_y (I' / (I' J' + J' I')) \bone_x$.

One easily sees that the residue classes of the elements of $R_{x, y}$ span $\bone_x (I' / (I' J' + J' I')) \bone_y$ as a vector space (see for example the proof of the corresponding statement in the proof of~\cite{MarSalSte2021}*{Theorem~4.1}). Thus assume the residue classes of the elements of $R_{x, y}$ are linearly dependent. This implies that there exists $\rho \in R_{x, y}$ such that $\rho \in K' + I' J' + J' I'$, where $K' := K / \langle R^{(0)} \rangle$ and $K := \langle R \setminus \{ \rho \} \rangle$. Consequently, $K' + I' J' + J' I' = I'$. Note that $I' J' + J' I'$ is the radical of $A'$ as an $A'$-$A'$-bimodule (here we use that $A$ is weakly triangular, hence there exists $n \in \bbN_+$ with $\sigma \in \langle R^{(0)} \rangle$, for each path $\sigma$ in $Q$ of length $n$). Thus the Nakayama lemma (see for example~\cite{AndFul1992}*{Corollary~15.13}) implies $I' = K'$, i.e., $I = K$, which contradicts the minimality of $R$.
\end{proof}

Before we continue we formulate an easy generalization of the results of~\cite{BobSch2019a}*{Subsection~5.5}.

\begin{lemm} \label{lemm:degone}
Let $Q$ be the quiver
\[
\xymatrix{
0 \ar@(ul,ur)^{\varepsilon_0}
&
1 \ar@(ul,ur)^{\varepsilon_1} \ar@<-1ex>[l]_{\alpha_{11}} \ar@<1ex>^{\alpha_{1n_1}}_{\cdots}[l]
&
2 \ar@(ul,ur)^{\varepsilon_2} \ar@<-1ex>[l]_{\alpha_{21}} \ar@<1ex>^{\alpha_{2n_2}}_{\cdots}[l]
&
\cdots \ar@<-1ex>[l]_{\alpha_{31}} \ar@<1ex>^{\alpha_{3n_3}}_{\cdots}[l]
&
k \ar@(ul,ur)^{\varepsilon_k} \ar@<-1ex>[l]_{\alpha_{k1}} \ar@<1ex>^{\alpha_{kn_k}}_{\cdots}[l]
},
\]
where $k, n_1, \ldots, n_k \in \bbN_+$. If $A'$ is a geometrically irreducible algebra, whose Gabriel quiver $Q'$ is a subquiver of $Q$, and $R'$ is a minimal set of relations such that $A' \simeq \Bbbk Q' / \langle R' \rangle$, then $\deg \rho \in \{ 0, 1 \}$, for each $\rho \in R'$.
\end{lemm}

\begin{proof}
We adapt arguments from the proof of~\cite{BobSch2019a}*{Corollary~5.5}: Let $R'^{(0)}$ be the set of relations of degree $0$ in $R'$. Without loss of generality, we may assume $R'^{(0)} := \{ \varepsilon_i^{m_i} : \text{$i \in Q_0''$} \}$, for some $m_i \geq 2$, $i \in Q_0''$, where $Q_0''$ is the set of $i \in [0, k]$ such that $\varepsilon_i \in Q_1'$. Choose $m_i \geq 2$, for each $i \in Q_0 \setminus Q_0''$, and let $R := R' \cup \{ \varepsilon_i^{m_i} : \text{$i \in Q_0 \setminus Q_0''$} \}$. If $A := \Bbbk Q / \langle R \rangle$, then $A'$ being geometrically irreducible implies $A$ is geometrically irreducible. Consequently, the claim follows from~\cite{BobSch2019a}*{Proposition~5.4}.
\end{proof}

The proof of the following fact uses ideas from the proof of~\cite{BobSch2019b}*{Lemma~5.1}.

\begin{prop} \label{prop:degone}
Let $A$ be a geometrically irreducible algebra with Gabriel quiver $Q$, and let $R$ be a minimal set of relations such that $A \simeq \Bbbk Q / \langle R \rangle$. If $\deg \rho \neq 1$, for each $\rho \in R$, then $\deg \rho = 0$, for each $\rho \in R$.
\end{prop}

\begin{proof}
Let $R^{(0)}$ be the set of relations of degree $0$ in $R$ and $I := \langle R^{(0)} \rangle$. Our aim is to show that $\rho \in I$, for each $\rho \in R \setminus R^{(0)}$.

For a path $\sigma = \alpha_1 \cdots \alpha_l$ in $Q$, let $\supp \sigma := \{ t (\alpha_1), s (\alpha_1), \ldots, s (\alpha_l) \}$. Note that $\deg \sigma = \# \supp \sigma - 1$, since $A$ is weakly triangular by Lemma~\ref{lemm:geoirr}\eqref{point:geoirr1}. If $\rho = \sum_{i = 1}^k \lambda_i \sigma_i$ is a relation in $R$, we write $\rho = \sum_{X \subseteq Q_0} \rho_X$, where $\rho_X := \sum_{i : \supp \sigma_i = X} \lambda_i \sigma_i$, for $X \subseteq Q_0$. By induction on $\# X$ we show $\rho_X \in I$, for each $\rho \in R \setminus R^{(0)}$. The claim is obvious if $\# X = 1$ (indeed, weak triangularity implies that $\rho_X = \rho \in R^{(0)}$, if $\# X = 1$ and $\rho_X \neq 0$) and there is nothing to prove if $\# X = 2$ (if $\# X = 2$ and $\rho_X \neq 0$, then $\deg \rho = 1$ by weak triangularity).

Now fix $X \subseteq Q_0$ with $\# X > 2$, and assume $\rho_Y \in I$, for each $\rho \in R \setminus R^{(0)}$ and $Y$ with $\#Y < \# X$. By replacing every $\rho \in R \setminus R^{(0)}$ by $\rho - \sum_{Y : \# Y < \# X} \rho_Y$ (note that $(\rho - \sum_{Y : \# Y < \# X} \rho_Y)_X = \rho_X$) we may assume
\renewcommand{\theequation}{$*$}
\begin{equation} \label{equation}
\text{$\rho_Y = 0$, for each $\rho \in R \setminus R^{(0)}$ and $Y$ with $\#Y < \# X$}.
\end{equation}

We number the elements $x_0$, \ldots, $x_l$ of $X$ in such a way that there are no paths from $x_i$ to $x_j$ in $Q$, for $i < j$ (here we use again that $A$ is weakly triangular). Let $Q_1'$ be the union of the set of all loops at $x_0$, \ldots, $x_l$ and the set of arrows $\alpha \in Q_1$ such that $s \alpha = x_i$ and $t \alpha = x_{i - 1}$, for some $i \in [1, l]$. If $Q' := (X, Q_1')$, then $\rho_X$ is a linear combination of paths in $Q'$, for each $\rho \in R \setminus R^{(0)}$. Let $A' := A / \langle  (\bone_{Q_0 \setminus X} \rangle + \langle Q_1 \setminus Q_1' \rangle)$. We show that $A'$ is geometrically irreducible.

First observe that $A' \simeq \Bbbk Q' / I'$, where $I'$ is the ideal in $\Bbbk Q'$ generated by the restrictions $\rho |_{Q'}$ of $\rho \in R$ to $Q'$. If $\rho \in R^{(0)}$, then either $\rho |_{Q'} = 0$ (if $s \rho \not \in X$) or $\rho |_{Q'} = \rho$ (if $s \rho \in X$). On the other hand, if $\rho \in R \setminus R^{(0)}$, then assumption~\eqref{equation} implies $\rho |_{Q'} = \rho_X$. Consequently, $I' = \langle R' \rangle$, where
\[
R' := \{ \rho \in R^{(0)} : \text{$s \rho \in X$} \} \cup \{ \rho_X : \text{$\rho \in R \setminus R^{(0)}$} \}.
\]

Now let $A'' := A / \langle  \bone_{Q_0 \setminus X} \rangle$. Then obviously $A''$ is geometrically irreducible. Moreover, $A'' \simeq \Bbbk Q'' / I''$, where $Q''$ is the full subquiver of $Q$ with the vertex set $X$ and $I''$ is the ideal in $\Bbbk Q''$ generated by the restrictions $\rho |_{Q''}$ of $\rho \in R$ to $Q''$. Obviously, $\rho |_{Q''} = \rho |_{Q'}$, if $\rho \in R^{(0)}$. Furthermore, if $\rho \in R \setminus R^{(0)}$, then $\rho |_{Q''} = \rho_X$, again by assumption~\eqref{equation}. In other words, $I'' = \langle R' \rangle$. Consequently, $\rep_{A''} (\bd)$ is the product of $\rep_{A'} (\bd)$ and an affine space, for each dimension vector $\bd$, hence $A''$ being geometrically irreducible implies $A'$ is geometrically irreducible as well.

Now we apply Lemma~\ref{lemm:degone} to $A'$ (note that $A'$ being geometrically irreducible implies that there is at most one loop at each vertex of $Q'$ by Lemma~\ref{lemm:geoirr}\eqref{point:geoirr2}) and conclude that $\rho_X$ belongs to the ideal generated by the relations in $R^{(0)} \cap R'$, for each $\rho \in R \setminus R^{(0)}$. Consequently, $\rho_X \in I$, for each $\rho \in R \setminus R^{(0)}$, and this finishes the proof.
\end{proof}

Now we are able to conclude a classification of the ext-irreducible algebras.

\begin{theo}
Up to Morita equivalence, an algebra $A$ is ext-irreducible if and only if $A$ is a simple loop extension of a hereditary algebra.
\end{theo}

Before giving the proof we need some preparations. Let $A$ be the path algebra of a bound quiver $(Q, I)$. For each loop $\alpha$, let $m_\alpha$ be the minimal $m \in \bbN_+$ such that $\alpha^m \in I$. A minimal set $R$ of relations generating $I$ is called normalized if, for each loop $\alpha$, $\alpha^{m_\alpha}$ is the unique element of $R$ which has a summand containing $\alpha^{m_\alpha}$ as a subpath (in particular, $\alpha^{m_\alpha} \in R$). In general a normalized minimal set of relations generating $I$ may not exist (this is, for example, the case for $A = \Bbbk [X, Y] / (X^2 - Y^2, X^4)$). However, it obviously exists if $A$ is a quotient of a simple loop extension of a hereditary algebra.

Now let $A = \Bbbk Q / I$ be a geometrically irreducible algebras. According to~Lemma~\ref{lemm:geoirr}\eqref{point:geoirr2}, for each vertex $x$ there is at most one loop $\alpha$ such that $s \alpha = x$. Moreover, $A$ is weakly triangular by Lemma~\ref{lemm:geoirr}\eqref{point:geoirr1}. Consequently, $A$ is a quotient of a simple loop extension of a hereditary algebra. In particular, there exists a normalized minimal set of relations generating $I$.

\begin{proof}
We already know from Lemma~\ref{prop:ext.irr} that the simple loop extensions of hereditary algebras are ext-irreducible. Thus we need to show that if $A = \Bbbk Q / I$ is ext-irreducible, then $A$ is a simple loop extension of a hereditary algebra. Recall from Proposition~\ref{prop:3.2} that $A$ is mono-irreducible. Moreover, $A$ is geometrically irreducible ($\rep_A (\bd) = \calE_A (\bzero, \bd)$).

Let $R$ be a normalized minimal set of generators of $I$. According to Proposition~\ref{prop:degone} it suffices to show $\deg \rho \neq 1$, for all $\rho \in R$. Suppose this is not the case and fix $\rho \in R$ with $\deg \rho = 1$. Put $x := t \rho$ and $y := s \rho$. If $B := A / \langle \bone_{Q_0 \setminus \{ x, y \}} \rangle$, then $B$ is also mono- and geometrically irreducible. Moreover, $B$ has exactly two simples, hence by Theorem~\ref{theo:main} (or~\cite{BobSch2019b}*{Theorem~1.1}) we may assume $B = A (n, m, 1)$ or $B = A (n, m, m - 1)$, for some $n \in \bbN_+$ and $m \geq 2$ (the cases of $n = 0$ and $B = A' (n, m_0, m_1)$ are excluded by Proposition~\ref{prop:Bongartz}, since from $R$ being normalized it follows that $\rho$ induces a relation of degree $1$ in $B$, which cannot be generated by relations of degree $0$). In order to treat the above cases simultaneously, let $l := 2$ in the former case, and $l := m$ in the latter one.

Let $\bd := (1, l)$ and $\be := (1, 1)$. Then $\calM_B (\be, \bd)$ consist of the tuples
\[
\xymatrix@R=3\baselineskip{%
k \ar@(ul,dl)_{[0]} \ar[d]_{[\lambda]} & k \ar@(ur,dr)^{[0]} \ar@<-1.5ex>_{[\mu_1]}[l] \ar@<1.5ex>^{[\mu_n]}_{\vdots}[l] \ar[d]^W
\\
k \ar@(ul,dl)_{[0]} & k^l \ar@(ur,dr)^{U} \ar@<-1.5ex>_{V_1}[l] \ar@<1.5ex>^{V_n}_{\vdots}[l]
}
\]
where $\mu_1, \ldots, \mu_n, \lambda \in \Bbbk$, $V_1, \ldots, V_n \in \bbM_{1 \times l} (\Bbbk)$, $U \in \bbM_{l \times l} (\Bbbk)$, $W \in \bbM_{l \times 1} (\Bbbk)$, such that
\[
V_1 U^{l - 1} = 0, \; U^l = 0, \; \lambda \mu_1 = V_1 W, \; \ldots, \lambda \mu_n = V_n W, \; U W = 0,
\]
and $\lambda$ and $W$ are nonzero.

Let $\calU_1$ be the set of the tuples as above such that $\rk U = l - 1$ and $\calU_2$ be set of the tuples such that $\mu_1 \neq 0$. They are easily seen to be nonempty open subsets of $\calM_B (\be, \bd)$. If $\calM_B (\be, \bd)$ were irreducible, $\calU_1$ and $\calU_2$ would have a nonempty intersection. However, $V_1 U^{l - 1} = 0$ and $U W = 0$, imply $\Im U^{l - 1} \subseteq \Ker V_1$ and $\Im W \subseteq \Ker U$. Moreover, if $\rk U = l - 1$, then $U^l = 0$ implies $\Ker U = \Im U^{l - 1}$, hence $\Im W \subseteq \Ker V_1$. Consequently, $\lambda \mu_1 = V_1 W = 0$. Since $\lambda \neq 0$, $\mu_1 = 0$, i.e., $\calU_1 \cap \calU_2 = \varnothing$. Thus $\calM_B (\be, \bd)$ is reducible, a contradiction.
\end{proof}

\bibsection

\end{document}